\documentclass[final,1p,times,numbers,sort&compress]{elsarticle}
\usepackage{txfonts}

\usepackage{amssymb}
 \usepackage{amsthm}
\usepackage{amscd}
\usepackage{amsmath}
\usepackage{amsfonts}
\usepackage{amssymb}
\usepackage{graphicx}
\usepackage{subfigure}
\usepackage{float}
\usepackage{epsfig}
\usepackage{graphics}

\numberwithin{equation}{section}

\newtheorem{theorem}{Theorem}[section]

\newtheorem{remark}[theorem]{Remark}
\newtheorem{lemma}[theorem]{Lemma}

\newtheorem{definition}{Definition}[section]

\journal{$\ast\ast\ast$}

\begin{document}

\begin{frontmatter}

\title{Convergence of ground state solutions for nonlinear Schr\"{o}dinger equations on graphs}

\author[label1]{Ning Zhang}
 \ead[label1]{nzhang@amss.ac.cn}
\address[label1]{China Institute for Actuarial Science,
 Central University of Finance and Economics, Beijing 100081, P. R. China}
\author[label2,label3]{Liang Zhao}
 \ead[label2]{ liangzhao@bnu.edu.cn}
\address[label2]{Corresponding author, School of Mathematical Sciences,
 Beijing Normal
University, Beijing 100875, P. R. China}
\address[label3]{Laboratory of Mathematics and
Complex Systems, Ministry of Education,
Beijing 100875, P. R. China}

\begin{abstract}
We consider the nonlinear Schr\"{o}dinger equation $-\Delta u+(\lambda a(x)+1)u=|u|^{p-1}u$ on a locally finite graph $G=(V,E)$. We prove via the Nehari method that if $a(x)$ satisfies certain assumptions, for any $\lambda>1$, the equation admits a ground state solution $u_\lambda$. Moreover, as $\lambda\rightarrow \infty$, the solution $u_\lambda$ converges to a solution of the Dirichlet problem $-\Delta u+u=|u|^{p-1}u$ which is defined on the potential well $\Omega$. We also provide a numerical experiment which solves the equation on a finite graph to illustrate our results.
\end{abstract}

\begin{keyword}
Sch\"{o}dinger equation\sep locally finite graph\sep ground state\sep potential well
\MSC 35A15 \sep 35Q55 \sep 58E30

\end{keyword}

\end{frontmatter}

\section{Introduction and main results}

Analysis on graphs is an important mathematical topic and has has many applications in other fields, such as imagine processing, data analysis, neural network, etc. \cite{chakik,curtis,desquesnes,elmoataz1,elmoataz2,elmoataz3,
elmoataz4}. Recently, many researchers have pay attention to the differential equations on graphs. There are works related to some important geometric inequalities on graphs \cite{bauerhorn,horn}. Many aspects about heat equation such as existence and non-existence of global solutions \cite{chunglee,linwu2,liu}, uniqueness and blow-up properties \cite{huang,xin}, estimates for heat kernel \cite{horn,linwu1,wojciechowski} have also been considered. For the elliptic case, Grigoryan et al \cite{grigor1,grigor2,grigor3} established existence results on graphs for some nonlinear elliptic equations based on the variational framework and this inspires us to deal with the Sch\"{o}dinger type equations on graphs.

The nonlinear Sch\"{o}dinger type equation of the form
\begin{equation*}
-\Delta u+b(x)u=f(x,u),\ \ u\in W^{1,2}(\Omega),
\end{equation*}
where $\Omega\subseteq {\mathbb{R}^n}$, $n\geq 2$, $f(x,s):\Omega\times \mathbb{R}\rightarrow \mathbb{R}$ is a nonlinear continuous function and $b(x)\in C(\Omega, \mathbb{R})$ is a given potential, has been extensively studied during the past several decades. It has attracted great interest because not only its importance in applications but also it provides a good model for developing new mathematical methods. Many papers are devoted to this kind of equations and among them, the readers can refer to \cite{brezisnirenberg,
cao,liwangzeng,rabinowitz,yang2,
yangzhao,zhaochang,zhaozhang} and the references therein.
For the Euclidean case, an important property about the Sch\"{o}dinger type equation is convergence of ground state solutions \cite{alves,bartschwang,clappding,dingtanaka,hezou,
wangzhou}. We may expect that this still holds for this kind of equations on graphs and explore the relationships between the structures of a graph and the partial differential equation.

Now let us explain our problems in details. Suppose $G=(V,E)$ be a graph which is locally finite and connected, where $V$ denotes the vertex set and $E$ denotes the edge set. Here we call $G$ a locally finite graph if for any $x \in V$, there are only finite $y\in V$ such that $xy\in E$. A graph is connected if any two vertexex $x$ and $y$ can be connected via finite edges. We use $w_{xy}>0$ to denote the weight of an edge $xy\in E$. If $w_{xy}=w_{yx}$ for any $xy\in E$, we call it a symmetric weight on $G$. The measure $\mu: V\rightarrow {\mathbb{R}^+}$ on the graph is a finite positive function on $G$. We call it a uniformly positive measure if there exists a constant $\mu_{\min}>0$ such that $\mu(x)\geq \mu_{\min}$ for all $x\in V$. Consider a domain $\Omega\subset V$. The distance $d(x,y)$ of two vertex $x, y\in \Omega$ is defined by the minimal number of edges which connect these two vertexes. If the distance $d(x,y)$ is uniformly bounded from above for any $x, y\in \Omega$, we call $\Omega$ a bounded domain in $V$. The boundary of $\Omega$ is defined as
\begin{equation}\label{boundary}
\partial \Omega:=\{y\not\in \Omega: \exists x\in \Omega \text{ such that } xy\in E\}
\end{equation}
and the interior of $\Omega$ is denoted by
$\Omega^\circ=\Omega\setminus \partial \Omega$. Obviously, we have that $\Omega^\circ=\Omega$ which is different with the Euclidean case.

To study partial differential equations on graphs, we define the $\mu$-Laplacian of a function $u: V\rightarrow \mathbb{R}$ as
\begin{equation}\label{laplacian}
\Delta u(x):=\frac{1}{\mu(x)}\sum_{y\sim x}w_{xy}(u(y)-u(x)),
\end{equation}
where $y\sim x$ stands for any vertex $y$ connected with $x$ by an edge $xy\in E$. We will call it Laplacian for brevity throughout this paper. The gradient form of two functions $u$ and $v$ on the graph is defined by
\begin{equation}\label{gradient}
\Gamma(u,v)(x):=\frac{1}{2\mu(x)}\sum_{y\sim x}w_{xy}(u(y)-u(x))(v(y)-v(x)).
\end{equation}
In particular, we use $\Gamma(u)$ to denote $\Gamma(u,u)$ and the length of the gradient for $u$ is
\begin{equation}\label{gradientlength}
|\nabla u|(x):=\sqrt{\Gamma(u)(x)}
=\left(\frac{1}{2\mu(x)}\sum_{y\sim x}w_{xy}(u(y)-u(x))^2\right)^{1/2}.
\end{equation}

The equation we are interested in is
\begin{equation}\label{equation}
-\Delta u+(\lambda a(x)+1)u=|u|^{p-1}u \ \ \text{in}\ \  V,
\end{equation}
where $2\leq p<\infty$. For the background and convergence of ground state solutions of (\ref{equation}) defined on Euclidean space, one can refer to \cite{bartschwang}. The potential $a(x)$ is a function defined on $V$ and we assume that $a(x)$ satisfies the following two conditions.\\

\noindent$(A_1)$ $a(x)\geq 0$ and the potential well $\Omega=\{x\in V: a(x)=0\}$ is a non-empty, connected and bounded domain in $V$.

\noindent$(A_2)$ There exists a vertex $x_0\in V$ such that
$a(x)\rightarrow +\infty$ as $d(x,x_0)\rightarrow +\infty$.\\

The integral of a function $u$ over $V$ is defined by
\begin{equation}\label{integral}
\int_V ud\mu=\sum_{x\in V}\mu(x)u(x).
\end{equation}
Let $C_c(V)$ be the set of all functions with compact support and $W^{1,2}(V)$ be the completion of $C_c(V)$ under the norm
\begin{equation}\label{norm}
\|u\|_{W^{1,2}(V)}=\left(\int_V(\nabla u^2+u^2) d\mu\right)^{1/2}.
\end{equation}
Clearly, $W^{1,2}(V)$ is a Hilbert space with the inner product
$$
<u,v>=\int_V (\Gamma(u,v)+uv)d\mu, \ \ \forall u,v\in W^{1,2}(V).
$$
To study the problem (\ref{equation}), it is natural to consider a function space
$$
E_\lambda=\left\{u\in W^{1,2}(V): \int_V \lambda a(x)u^2d\mu<+\infty\right\}
$$
with the norm
$$
\|u\|_{E_\lambda}=\left(\int_V (|\nabla u|^2+(\lambda a(x)+1)u^2)d\mu\right)^{1/2}.
$$
The space $E_\lambda$ is also a Hilbert space and its inner product is
$$
<u,v>_{E_\lambda}=\int_V (\Gamma(u,v)+(\lambda a(x)+1)uv)d\mu, \ \ \forall u,v\in E_\lambda.
$$
The functional related to (\ref{equation}) is
\begin{equation}\label{functional}
J_{\lambda}(u)=\frac{1}{2}\int_V (|\nabla u|^2+(\lambda a(x)+1)u^2)d\mu-\frac{1}{p+1}\int_V |u|^{p+1}d\mu.
\end{equation}
The Nehari manifold related to (\ref{equation}) is defined as
$$
\mathcal{N}_\lambda:=\left\{u\in E_\lambda\setminus \{0\}: J'_\lambda(u)\cdot u=0\right\}.
$$
Namely,
\begin{equation}\label{nehari}
\mathcal{N}_\lambda=\left\{u\in E_\lambda\setminus \{0\}: \int_V|\nabla u|^2+(\lambda a+1)u^2d\mu=\int_V |u|^{p+1} d\mu\right\}.
\end{equation}
Let $m_\lambda$ be
\begin{equation}\label{infnehari}
m_\lambda:=\inf_{u\in\mathcal{N}_\lambda}J_\lambda(u).
\end{equation}
If $m_\lambda$ can be achieved by some function $u_\lambda\in \mathcal{N}_\lambda$, $u_\lambda$ shall have the the least energy among all functions belong to the Nehari manifold and in fact, it is a critical point of the functional $J_\lambda$. We call $u_\lambda$ a ground state solution of (\ref{equation}).
Via the method first developed by Z. Nehari in \cite{nehari}, we prove that

\begin{theorem}\label{existence}
Let $G=(V,E)$ be a locally finite and connected graph with symmetric weight and uniformly positive measure. Assume $a(x): V\rightarrow [0,+\infty)$ is a function satisfying $(A_1)$ and $(A_2)$. Then for any positive constant $\lambda>1$, there exists a ground state solution $u_\lambda$ of the equation (\ref{equation}).
\end{theorem}

To study the behavior of $u_\lambda$ as $\lambda\rightarrow \infty$, we introduce the Dirichlet problem
\begin{align}\label{dirichlet}
\begin{cases}
-\Delta u+u=|u|^{p-1}u \  &\text{in}\  \Omega;\\
u=0,\  &\text{on} \  \partial \Omega.
\end{cases}
\end{align}
It is suitable to study (\ref{dirichlet}) in the Hilbert space $W_0^{1,2}(\Omega)$ which is the completion of $C_c(\Omega)$ under the norm
$$
\|u\|_{W^{1,2}_0(\Omega)}=\left(\int_{\Omega\cup \partial \Omega}|\nabla u|^2 d\mu
+\int_{\Omega}u^2d\mu\right)^{\frac{1}{2}}.
$$
Since the formula for integral by parts which will be proved in Section 2(Lemma \ref{integral1}), in the definition of the norm for $W_0^{1,2}(\Omega)$, we need an additional integral on $\partial \Omega$ for the gradient form of $u$ which is different with the Euclidean case. The functional related to (\ref{dirichlet}) is
\begin{equation}\label{ofunctional}
J_\Omega(u)=\frac{1}{2}\left(\int_{\Omega\cup\partial \Omega}|\nabla u|^2d\mu+\int_{\Omega}u^2d\mu\right)-\frac{1}{p+1}\int_\Omega |u|^{p+1}d\mu.
\end{equation}
The corresponding Nehari manifold is
\begin{equation}\label{dnehari}
\mathcal{N}_\Omega:=\left\{u\in W_0^{1,2}(\Omega)\setminus \{0\}: \|u\|_{W^{1,2}_0(\Omega)}^2=\|u\|_{L^{p+1}(\Omega)}^{p+1}\right\}.
\end{equation}
Similar to (\ref{equation}), the Dirichlet problem also has a ground state solution.

\begin{theorem}\label{oexistence}
Suppose $G=(V,E)$ be a locally finite and connected graph with symmetric weight and uniformly positive measure. Let $\Omega$ be a non-empty, connected and bounded domain in $V$. Then the equation (\ref{dirichlet}) has a ground state solution $u_0\in W_0^{1,2}(\Omega)$.
\end{theorem}

In fact, (\ref{dirichlet}) is some kind of limit problem for (\ref{equation}) as $\lambda\rightarrow +\infty$, where $\Omega$ is the potential well, namely $\Omega=\{x\in V: a(x)=0\}$. More precisely, we have

\begin{theorem}\label{convergence}
Under the same assumptions as in Theorem \ref{existence}, we have that, for any sequence $\lambda_k\rightarrow \infty$, up to a subsequence, the corresponding ground state solutions $u_{\lambda_k}$ of (\ref{equation}) converges in $W^{1,2}(V)$ to a ground state solution of (\ref{dirichlet}).
\end{theorem}

We organize the rest of the paper as follows. In Section 2, we introduce some known results and basic properties which are useful for the proof of our results. In Section 3, we give proofs of the main results. In Section 4, we carry out numerical experiments on a finite graph and use the numerical solutions to illustrate the convergence of $u_\lambda$ as $\lambda\rightarrow \infty$.

\section{Preliminaries}

From now on , we always assume that $G=(V,E)$ is a locally finite and connected graph with symmetric weight and uniformly positive measure. In this section, we shall introduce some basic analytical properties on graphs and prove some compactness results related to $E_\lambda$ and $W_0^{1,2}(\Omega)$.

First we present two lemmas about integral by parts on graphs.

\begin{lemma}\label{integral}
Suppose that $u\in W^{1,2}(V)$ and its Laplacian $\Delta u$ is well defined. Then for any $v\in W^{1,2}(V)$ we have
\begin{equation}\label{part1}
\int_V \nabla u\cdot\nabla vd\mu=\int_V \Gamma(u,v)d\mu=-\int_V \Delta u v  d\mu.
\end{equation}
\end{lemma}

\begin{proof}
By direct computations, we have
\begin{eqnarray*}
\int_V \Gamma(u,v)d\mu&=&\int_V\frac{1}{2\mu(x)}\sum_{y\sim x}w_{xy}(u(y)-u(x))(v(y)-v(x))d\mu\\
&=&\frac{1}{2}\sum_{x\in V}\sum_{y\sim x}w_{xy}(u(y)-u(x))(v(y)-v(x))\\
&=&\frac{1}{2}\sum_{x\in V}\sum_{y\sim x}w_{xy}(u(y)-u(x)))v(y)-\frac{1}{2}\sum_{x\in V}\sum_{y\sim x}w_{xy}(u(y)-u(x)))v(x)\\
&=&\frac{1}{2}\sum_{y\in V}\sum_{x\sim y}w_{xy}(u(y)-u(x)))v(y)-\frac{1}{2}\sum_{x\in V}\sum_{y\sim x}w_{xy}(u(y)-u(x)))v(x)\\
&=&-\sum_{x\in V}\sum_{y\sim x}w_{xy}(u(y)-u(x)))v(x)\\
&=&-\int_V \Delta u vd\mu.
\end{eqnarray*}
\end{proof}

When we integrate by parts on a bounded domain of a graph, there shall be a boundary term even if $v$ has a compact support. Therefore the formula becomes different with the Euclidean case. Precisely, we have

\begin{lemma}\label{integral1}
Suppose that $u\in W^{1,2}(V)$ and its Laplacian $\Delta u$ is well defined. Let $v$ be a function belongs to $C_c(\Omega)$, where $\Omega\subset V$ is a bounded domain. Then we have
\begin{equation}\label{part1}
\int_{\Omega\cup \partial \Omega} \nabla u\cdot\nabla vd\mu=\int_{\Omega\cup \partial \Omega} \Gamma(u,v)d\mu=-\int_{\Omega} \Delta u v  d\mu.
\end{equation}
\end{lemma}

\begin{proof}
Similar to the proof of Lemma \ref{integral}, we have
\begin{eqnarray*}
\int_{\Omega\cup \partial \Omega} \Gamma(u,v)d\mu
&=&\frac{1}{2}\sum_{x\in \Omega\cup \partial \Omega}\sum_{y\sim x}w_{xy}(u(y)-u(x))(v(y)-v(x))\\
&=&\frac{1}{2}\sum_{x\in \Omega\cup \partial \Omega}\sum_{y\sim x}w_{xy}(u(y)-u(x)))v(y)-\frac{1}{2}\sum_{x\in \Omega\cup \partial \Omega}\sum_{y\sim x}w_{xy}(u(y)-u(x)))v(x)\\
&=&\frac{1}{2}\sum_{y\in \Omega}\sum_{x\sim y}w_{xy}(u(y)-u(x)))v(y)-\frac{1}{2}\sum_{x\in \Omega}\sum_{y\sim x}w_{xy}(u(y)-u(x)))v(x)\\
&=&-\sum_{x\in \Omega}\sum_{y\sim x}w_{xy}(u(y)-u(x)))v(x)\\
&=&-\int_\Omega \Delta u vd\mu.
\end{eqnarray*}
\end{proof}

Due to Lemma \ref{integral}, we can define the weak solution the equation (\ref{equation}) as

\begin{definition}\label{weaksolution}
Suppose $u\in E_\lambda$. If for any $\varphi\in C_c(V)$, there holds
$$
\int_V(\Gamma(u,\varphi)+(\lambda a+1)u\varphi)d\mu
=\int_V |u|^{p-1}u \varphi d\mu,
$$
$u$ is called a weak solution of (\ref{equation}).
\end{definition}

According to Lemma \ref{integral1}, we define the weak solution of (\ref{dirichlet}) as

\begin{definition}\label{weaksolution1}
Suppose $u\in W^{1,2}_0(\Omega)$. If for any $\varphi\in C_c(\Omega)$, there holds
$$
\int_{\Omega\cup \partial \Omega}\Gamma(u,\varphi)d\mu+\int_{\Omega}u\varphi d\mu
=\int_\Omega |u|^{p-1}u \varphi d\mu,
$$
$u$ is called a weak solution of (\ref{equation}).
\end{definition}

\begin{remark}\label{pointwise}
If u is a weak solution of (\ref{dirichlet}), by Lemma \ref{integral1}, for any test function $\varphi\in C_c(\Omega)$, we have
$$
\int_\Omega(-\Delta u \varphi+u\varphi)d\mu
=\int_\Omega |u|^{p-1}u \varphi d\mu.
$$
For any fixed $x_0\in \Omega$, take the test function to be
\begin{align*}
\varphi_0(x)=\begin{cases}
-\Delta u(x_0)+u(x_0)-|u|^{p-1}u(x_0)\ \ & x=x_0;\\
0\ \ &x\not=x_0.
\end{cases}
\end{align*}
Then we get
$$
-\Delta u(x_0)+u(x_0)=|u|^{p-1}u(x_0),
$$
which tells us that, in fact, a weak solution is also a point-wise solution. For a weak solution of (\ref{equation}), we can prove the result similarly and we omit it here.
\end{remark}

Finally in this section, we present some results about the compactness of $E_\lambda$ and $W^{1,2}_0(\Omega)$. Since we are concerned with the behaviors of solutions $u_\lambda$ of (\ref{equation}) as $\lambda\rightarrow \infty$, without loss of generality, we can assume that $\lambda>1$ in the next lemma. We use $\|\cdot\|_{q,V}$ and $\|\cdot\|_{q,\Omega}$ to denote the $L^q$ norms on $V$ and $\Omega$ respectively and we may omit the subscripts $V$ and $\Omega$ if it can be understood from the context. The next two lemmas are similar to results in \cite{grigor1,grigor2}.

\begin{lemma}\label{compact1}
Assume that $\lambda>1$ and $a(x)$ satisfies $(A_1)$ and $(A_2)$. Then $E_\lambda$ is continuously embedded into $L^q(V)$ for any $q\in [2,\infty)$ and the embedding is independent of $\lambda$. Namely, there exists a constant $C$ depending only on $q$ such that for any $u\in E_\lambda$,
$$
\|u\|_{q,V}\leq C\|u\|_{E_\lambda}.
$$
Moreover, for any bounded sequence $\{u_k\}\subset E_\lambda$, there exists $u\in E_\lambda$ such that, up to a subsequence,
\begin{align*}
\begin{cases}
u_k\rightharpoonup u\ &\text{in}\  E_\lambda;\\
u_k(x)\rightarrow u(x)\ &\forall x\in V;\\
u_k\rightarrow u\ &\text{in}\  L^{q}(V).
\end{cases}
\end{align*}
\end{lemma}

\begin{proof}
Suppose $u\in E_\lambda$. At any vertex $x_0\in V$, we have
\begin{eqnarray*}
\|u\|_{E_\lambda}^2&=&\int_V (|\nabla u|^2+(\lambda a+1)u^2)d\mu\\
&\geq&\int_V u^2d\mu\\
&=&\sum_{x\in V}\mu(x)u^2(x)\\
&\geq&\mu_{\min} u^2(x_0),
\end{eqnarray*}
which gives
$$
u(x_0)\leq \sqrt{\frac{1}{\mu_{\min}}}\|u\|_{E_\lambda}.
$$
Thus $E_\lambda\hookrightarrow L^{\infty}(V)$ continuously and the embedding is independent of $\lambda$. Then the interpolation gives continuous embedding $E_\lambda \hookrightarrow L^q(V)$ for any $2\leq q\leq \infty$.

Since $E_\lambda$ is a Hilbert space, for $\{u_k\}$ bounded in $E_\lambda$, we have that, up to a subsequence, $u_k\rightharpoonup u$ in $E_\lambda$. On the other hand, $\{u_k\}\subset E_\lambda$ is also bounded in $L^2(V)$ and we have the weak convergence in $L^2(V)$, which tells that, for any $\phi \in L^2(V)$,
\begin{equation}\label{wconverge}
\lim_{k\rightarrow \infty}\int_V (u_k-u) \phi d\mu
=\lim_{k\rightarrow \infty}\sum_{x\in V} \mu(x)(u_k-u)(x)\phi(x)=0,
\end{equation}
Take any $x_0\in V$ and let
\begin{align*}
\phi_0(x)=\begin{cases}
1\  &x=x_0;\\
0\  &x\neq x_0.
\end{cases}
\end{align*}
Obviously it belongs to $L^2(V)$. By substituting $\phi_0$ into (\ref{wconverge}) we get
$$
\lim_{k\rightarrow \infty} \mu(x_0)(u_k-u)(x_0)=0,
$$
which implies that $\lim_{k\rightarrow \infty}u_k(x)=u(x) $ for any $x\in V$.

Suppose $u_k\rightharpoonup u$ in $E_\lambda$. Because of the boundedness of $\{u_k\}$, there exists some constant $C_0$ such that
$$
\|u_k-u\|_{E_\lambda}^2\leq C_0.
$$
Since $a(x)$ satisfies $(A_2)$, $\forall \epsilon>0$, there exists some constant $R>0$ such that when $\text{dist}(x,x_0)>R$,
$$
a(x)\geq \frac{C_0}{\epsilon}.
$$
Noticing $\lambda>1$, we have
\begin{eqnarray}\label{lp1}
\int_{\text{dist}(x,x_0)>R} |u_k-u|^2d \mu
&\leq& \frac{\epsilon}{C_0}\int_{\text{dist}(x,x_0)>R} a|u_k-u|^2d \mu\nonumber\\
&\leq& \frac{\epsilon}{C_0}\|u_k-u\|^2_{E_\lambda}\nonumber\\
&\leq& \epsilon.
\end{eqnarray}
On the other hand, since $\{x\in V: \text{dist}(x,x_0)\leq R\}$ is a finite set and $u_k(x)\rightarrow u(x)$ for any $x\in V$ as $k\rightarrow \infty$, we have
$$
\lim_{k\rightarrow \infty}\int_{\text{dist}(x,x_0)\leq R} |u_k-u|^2d \mu=0.
$$
This together with (\ref{lp1}) gives that
$$
\lim_{k\rightarrow \infty}\int_{V} |u_k-u|^2d \mu=0.
$$
When $p=\infty$, we have
$$
\|u_k-u\|_{L^{\infty}(V)}^2\leq \frac{1}{\mu_{\min}}\int_V |u_k-u|^2d\mu\rightarrow 0\  \text{as}\  k\rightarrow \infty.
$$
Finally, for any $2<p<\infty$,
$$
\int_V |u_k-u|^pd\mu\leq \|u_k-u\|_{L^{\infty}(V)}^{p-2}\int_V |u_k-u|^2d\mu\rightarrow 0\  \text{as}\  k\rightarrow \infty.
$$
\end{proof}

For the space $W_0^{1,2}(\Omega)$, we also have

\begin{lemma}\label{compact2}
Assume that $\Omega$ is a bounded domain in $V$. Then $W_0^{1,2}(\Omega)$ is continuously embedded into $L^q(\Omega)$ for any $q\in [1,\infty]$. Namely, there exists a constant $C$ depending only on $q$ and $\Omega$ such that for any $u\in W_0^{1,2}(\Omega)$,
$$
\|u\|_{q,\Omega}\leq C\|u\|_{W_0^{1,2}(\Omega)}.
$$
Moreover, for any bounded sequence $\{u_k\}\subset W_0^{1,2}(\Omega)$, there exists $u\in W_0^{1,2}(\Omega)$ such that, up to a subsequence,
\begin{align*}
\begin{cases}
u_k\rightharpoonup u\ &\text{in}\  W_0^{1,2}(\Omega);\\
u_k(x)\rightarrow u(x)\ &\forall x\in \Omega;\\
u_k\rightarrow u\ &\text{in}\  L^{q}(\Omega).
\end{cases}
\end{align*}
\end{lemma}

\begin{proof}
The proof is almost the same as that of Lemma \ref{compact1}. The only difference is that $\Omega$ is a finite set now. By using this fact, it is easy to prove $\|u\|_{q,\Omega}\leq C\|u\|_{W_0^{1,2}(\Omega)}$ for any $q\in [1,\infty]$, where $C$ is a constant depending on $q$ and $\Omega$.
\end{proof}

\section{Existence of a ground state solution}

The existence of a ground state solution for (\ref{equation}) can be proved by standard variational methods. First we give several properties of $\mathcal{N}_\lambda$ and prove that $m_\lambda$ can be achieved in $\mathcal{N}_\lambda$.

\begin{lemma}\label{nonempty}
$\mathcal{N}_\lambda$ is non-empty.
\end{lemma}

\begin{proof}
$\forall u\in E_\lambda\setminus \{0\}$, we define
$$
f(t)=J_\lambda '(tu)\cdot tu=t^2\int_V|\nabla u|^2+(\lambda a+1)u^2d\mu-t^{p+1}\int_V |u|^{p+1}d\mu.
$$
Since $p\geq 2$ and $u\not\equiv 0$, there exists $t_0\in (0,\infty)$ such that $f(t_0)=0$ which implies that $t_0u\in \mathcal{N}_\lambda$.
\end{proof}

\begin{lemma}\label{lowerbound}
$m_\lambda=\inf_{u\in\mathcal{N}_\lambda}J_\lambda(u)>0$
\end{lemma}

\begin{proof}
If $u\in \mathcal{N}_\lambda$, we have
$$
\int_V|\nabla u|^2+(\lambda a+1)u^2d\mu=\int_V |u|^{p+1} d\mu.
$$
By Lemma \ref{compact1}, we have
$$
\|u\|_{E_\lambda}^2=\|u\|_{L^{p+1}}^{p+1}\leq C\|u\|_{E_\lambda}^{p+1},
$$
where $C>0$ is the constant of the embedding $E_{\lambda}\hookrightarrow L^{p+1}(V)$ and is independent of $\lambda$. Since $p\geq 2$, we get
$$
\|u\|_{E_\lambda}\geq \left(\frac{1}{C}\right)^{\frac{1}{p-1}}>0.
$$
This gives that
$$
m_\lambda=\inf_{u\in\mathcal{N}_\lambda}J_\lambda(u)
=\left(\frac{1}{2}-\frac{1}{p+1}\right)
\inf_{u\in\mathcal{N}_\lambda}\|u\|_{E_\lambda}^2
\geq \left(\frac{1}{2}-\frac{1}{p+1}\right)
\left(\frac{1}{C}\right)^{\frac{1}{p-1}}>0.
$$
\end{proof}

\begin{lemma}\label{minimize}
$m_{\lambda}$ can be achieved by some $u_\lambda\in \mathcal{N}_\lambda$. Namely, there exists some $u_\lambda\in \mathcal{N}_\lambda$ such that $J_\lambda(u_\lambda)=m_{\lambda}$.
\end{lemma}

\begin{proof}
Take a sequence $\{u_k\}\subset \mathcal{N}_\lambda$ such that $\lim_{k\rightarrow \infty}J_\lambda(u_k)=m_{\lambda}$. Since
$$
o_k(1)+m_\lambda=J_\lambda(u_k)=\frac{p-1}{2(p+1)}\|u_k\|_{E_\lambda}^2,
$$
where $\lim_{k\rightarrow \infty}o_k(1)=0$, we have that $\{u_k\}$ is bounded in $E_\lambda$. By Lemma \ref{compact1}, we can assume that there exists some $u_\lambda\in E_\lambda$ such that, as $k\rightarrow\infty$,
\begin{align*}
\begin{cases}
u_k\rightharpoonup u_{\lambda}\ &\text{in}\  E_\lambda;\\
u_k\rightarrow u_{\lambda}\ &\text{in}\  L^{p+1}(V);\\
u_k\rightarrow u_{\lambda}\ &\text{for any}\  x\in V.
\end{cases}
\end{align*}
By weak lower semi-continuity of the norm for $E_\lambda$ and convergence of $u_k$ to $u_\lambda$ in $L^{p+1}(V)$, we have
\begin{eqnarray}\label{minimize1}
J_\lambda(u_\lambda)&=&\frac{1}{2}\|u_\lambda\|_{E_\lambda}^2
-\frac{1}{p+1}\|u_\lambda\|_{p+1}^{p+1}\nonumber\\
&\leq& \liminf_{k\rightarrow \infty}
\left(\frac{1}{2}\|u_k\|_{E_\lambda}^2
-\frac{1}{p+1}\|u_k\|_{p+1}^{p+1}\right)\nonumber\\
&=&\liminf_{k\rightarrow \infty}J_\lambda(u_k)\nonumber\\
&=&m_\lambda.
\end{eqnarray}
Now to prove the lemma, we only need to show that $u_\lambda\in\mathcal{N}_\lambda$.

Up to a subsequence, we assume that $\|u_k\|_{E_\lambda}^2\rightarrow C>0$ for some positive constant $C$ as $k\rightarrow \infty$. This together with $\|u_k\|_{E_\lambda}^2=\|u_k\|_{p+1}^{p+1}$ gives that
$$
\|u_\lambda\|_{p+1}^{p+1}=\lim_{k\rightarrow \infty}\|u_k\|_{p+1}^{p+1}=C.
$$
Noticing that $u_k\in \mathcal{N}_\lambda$, we have
$$
\|u_\lambda\|_{E_\lambda}^2\leq \liminf_{k\rightarrow\infty}\|u_k\|_{E_\lambda}^2
=\liminf_{k\rightarrow\infty}\|u_k\|_{p+1}^{p+1}
=\|u_\lambda\|_{p+1}^{p+1}.
$$
If $\|u_\lambda\|^2<\|u_\lambda\|_{p+1}^{p+1}$, by similar arguments as in Lemma \ref{nonempty}, we get that there exists some $t\in (0,1)$ such that $tu_\lambda\in \mathcal{N}_\lambda$. This gives that
\begin{eqnarray*}
0<m_\lambda\leq J_\lambda(tu_\lambda)
&=&\left(\frac{1}{2}-\frac{1}{p+1}\right)\|tu_\lambda\|_{E_\lambda}^2\\
&=&t^2\left(\frac{1}{2}-\frac{1}{p+1}\right)\|u_\lambda\|_{E_\lambda}^2\\
&\leq&t^2\liminf_{k\rightarrow \infty}
\left(\frac{1}{2}-\frac{1}{p+1}\right)\|u_k\|_{E_\lambda}^2\\
&=&t^2\liminf_{k\rightarrow \infty}J_\lambda(u_k)\\
&=&t^2 m_\lambda<m_\lambda,
\end{eqnarray*}
which is a contradiction with the fact that $m_\lambda=\inf_{u\in\mathcal{N}_\lambda}J_\lambda(u)$. Therefore, $\|u_\lambda\|^2=\|u_\lambda\|_{p+1}^{p+1}$ and (\ref{minimize1}) gives that $m_\lambda$ is achieved by $u_\lambda\in \mathcal{N}_\lambda$.
\end{proof}

\begin{lemma}\label{criticalpoint}
$u_\lambda\in \mathcal{N}_\lambda$ is a ground state solution of (\ref{equation}).
\end{lemma}

\begin{proof}
We shall prove that for any $\phi\in C_c(V)$, we have
\begin{equation}\label{variation}
J'_\lambda(u_\lambda)\cdot \phi=0.
\end{equation}
We can choose a constant $\epsilon>0$, such that $u_\lambda+s\phi\not\equiv 0$ when $s\in (-\epsilon,\epsilon)$. For any $s\in (-\epsilon,\epsilon)$, there exists some $t(s)\in (0,\infty)$ such that $t(s)(u_\lambda+s\phi)\in \mathcal{N}_\lambda$. In fact, we can take
$$
t(s)=\left(\frac{\|u_\lambda+s\phi\|_{E_\lambda}^2}
{\|u_\lambda+s\phi\|_{p+1}^{p+1}}\right)^{\frac{1}{p+1}}.
$$
In particular, we have $t(0)=1$. Define a function $\gamma(s):(-\epsilon,\epsilon)\rightarrow \mathbb{R}$ as
$$
\gamma(s):=J_\lambda(t(s)(u_\lambda+s\phi)).
$$
Since $t(s)(u_\lambda+s\phi)\in \mathcal{N}_\lambda$ and $J_\lambda(u_\lambda)=\inf_{u\in \mathcal{N}_\lambda}J(u)$, we get that $\gamma(s)$ achieves its minimum at $s=0$. This implies that
\begin{eqnarray*}
0=\gamma'(0)&=&J'_\lambda(t(0)u)\cdot[t'(0)u_\lambda+t(0)\phi]\\
&=&J'_\lambda(u_\lambda)\cdot t'(0)u_\lambda+J'_\lambda(u_\lambda)\cdot\phi\\
&=&J'_\lambda(u_\lambda)\cdot\phi.
\end{eqnarray*}
\end{proof}

\begin{remark}
Just as what we have done for (\ref{equation}), we can also get a ground state solution $u_0$ for (\ref{dirichlet}) which achieves the minimum $m_\Omega$ of the functional $J_\Omega$ in ${\mathcal{N}}_\Omega$. Thus Theorem \ref{existence} and \ref{oexistence} are proved.
\end{remark}

\begin{remark}\label{positive}
As in \cite{rabinowitz}, we can replace the nonlinearity $|u|^{p-1}u$ on the right hand side of (\ref{equation}) and (\ref{dirichlet}) by $u_{+}^p$ to get positive solutions of the equations, where $u_{+}=\max\{u,0\}$.
\end{remark}

\begin{lemma}\label{bound}
There exists $\sigma>0$ such that, for any critical point $u\in E_\lambda\setminus \{0\}$ of $J_\lambda$, we have $\|u\|_{E_\lambda}\geq \sigma$. Here $\sigma$ is independent of $\lambda$.
\end{lemma}

\begin{proof}
Lemma \ref{compact1} tells us that
$$
\|u\|_{{p+1}}\leq C\|u\|_{E_\lambda},
$$
where $C$ is independent of $\lambda$. Since $u$ is a critical point of $J_\lambda$, we have that
\begin{eqnarray*}
0=J'_\lambda(u)\cdot u&=&\int_V (|\nabla u|^2+ (\lambda a+1)u^2)d\mu-\int_V |u|^{p+1} d\mu\\
&\geq& \|u\|_{E_\lambda}^2-C^{p+1}\|u\|_{E_\lambda}^{p+1}.
\end{eqnarray*}
Then we have
$$
\|u\|_{E_\lambda}\geq \left(\frac{1}{C}\right)^{\frac{p+1}{p-1}}
$$
and we can choose $\sigma=\left(\frac{1}{C}\right)^{\frac{p+1}{p-1}}$
\end{proof}

\begin{lemma}\label{bound1}
There exists $C_1>0$ which is independent of $\lambda$ such that if $\{u_k\}$ is a $(PS)_c$ sequence of $J_\lambda$, then
\begin{equation}\label{bound2}
\limsup_{k\rightarrow \infty}\|u_k\|_{E_\lambda}^2\leq \frac{2pc}{p-2}
\end{equation}
and either $c\geq C_1$ or $c=0$.
\end{lemma}

\begin{proof}
Since $J_\lambda(u_k)\rightarrow c$ and $J'_\lambda(u_k)\rightarrow 0$ as $k\rightarrow \infty$, we have
\begin{eqnarray*}
c+o_k(1)\|u_k\|_{E_\lambda}
&=&\limsup_{k\rightarrow \infty}\left(J_{\lambda}(u_k)-\frac{1}{p}J_{\lambda}'(u_k)u_k\right)\\
&=&\limsup_{k\rightarrow \infty}\left[\left(\frac{1}{2}-\frac{1}{p}\right)
\|u_k\|_{E_\lambda}^2
+\left(\frac{1}{p}-\frac{1}{p+1}\right)\|u_k\|_{L^{p+1}}^{p+1}\right]\\
&\geq& \limsup_{k\rightarrow \infty}\left(\frac{1}{2}-\frac{1}{p}\right)
\|u_k\|_{E_\lambda}^2\\
&=& \frac{p-2}{2p}\limsup_{k\rightarrow \infty}\|u_k\|_{E_\lambda}^2
\end{eqnarray*}
which gives (\ref{bound2}).

For any $u\in E_\lambda$,
\begin{eqnarray*}
J'_\lambda(u)\cdot u
&=&\|u\|_{E_\lambda}^2-\|u\|_{L^{p+1}}^{p+1}\\
&\geq& \|u\|_{E_\lambda}^2-C\|u\|_{E_\lambda}^{p+1},
\end{eqnarray*}
where we have used the embedding $E_\lambda\hookrightarrow L^{p+1}(V)$. Take $\epsilon=\left(\frac{1}{2C}\right)^{\frac{1}{p-1}}>0$. If $\|u\|_{E_\lambda}\leq \epsilon$, we have
$$
\frac{1}{2}\|u\|_{E_\lambda}^2\leq J'_\lambda(u)\cdot u.
$$
Take $C_1=\frac{\epsilon^2(p-2)}{2p}$ and suppose $c<C_1$. Since $\{u_k\}$ is a $(PS)_c$ sequence, (\ref{bound2}) gives
$$
\limsup_{k\rightarrow \infty} \|u_k\|_{E_\lambda}^2\leq
\frac{2pc}{p-2}<\epsilon^2.
$$
Hence, for $k$ large, we have
$$
\frac{1}{2}\|u_k\|_{E_\lambda}^2\leq J'_\lambda(u_k)\cdot u_k=o_k(1)\|u_k\|_{E_\lambda}.
$$
Then we have $\|u_k\|_{E_\lambda}\rightarrow 0$ as $k\rightarrow \infty$ which gives $J_\lambda(u_k)\rightarrow c=0$ and the desired results are proved for $C_1=\frac{\epsilon^2(p-2)}{2p}
=\left(\frac{1}{2C}\right)^{\frac{2}{p-1}}\frac{p-2}{2p}$.
\end{proof}

\begin{remark}\label{psbound}
By the proof of the existence of a ground state solution $u_\lambda$, we know that there exists a $(PS)_c$ sequence $u_k$ converges weakly to $u_\lambda$ in $E_\lambda$, where $c=m_\lambda$. By weak lower semi-continuity of the norm $\|\cdot\|_\lambda$, we get that $\|u_\lambda\|_{E_\lambda}$ is bounded by $\frac{2pm_\lambda}{p-2}$.
\end{remark}

\section{Proof of the main results}

Firstly, for the ground states $m_\lambda$ and $m_\Omega$, we have

\begin{lemma}\label{limit}
$m_\lambda\rightarrow m_\Omega$ as $\lambda\rightarrow \infty$.
\end{lemma}

\begin{proof}
It is obvious that $m_\lambda<m_\Omega$ for any $\lambda>0$. Because otherwise we can find a nontrivial solution $u_0$ of (\ref{equation}) which vanishes outside $\Omega$. This is impossible due to the maximum principle. Take a sequence $\lambda_k\rightarrow \infty$ such that
$$
\lim_{k\rightarrow \infty} m_{\lambda_k}=M\leq m_{\Omega}
$$
where $m_{\lambda_k}$ is the ground state of the the ground state solution $u_{\lambda_k}\in \mathcal{N}_{\lambda_k}$ of (\ref{equation}). Then Lemma \ref{bound1} tells us that $M>0$. Since $\{u_{\lambda_k}\}$ is uniformly bounded in $W^{1,2}(V)$, up to a subsequence, we assume that there exists some $u_0\in W^{1,2}(V)$ such that
$$
u_k\rightharpoonup u_0 \ \text{in}\  W^{1,2}(V)
$$
and for any $q\in [2,\infty)$,
$$
u_k\rightarrow u_0 \ \text{in}\  L^q(V).
$$

We claim that $u_0|_{\Omega^c}=0$. If not, there exists a vertex $x_0\not\in \Omega$ such that $u_0(x_0)\neq 0$. Since $u_{\lambda_k}\in \mathcal{N}_{\lambda_k}$, we have
$$
J_{\lambda_k}(u_{\lambda_k})=\frac{p-1}{2(p+1)}\|u_{\lambda_k}\|_{E_{\lambda_k}}^2\geq \frac{p-1}{2(p+1)}\lambda_k \int_V a(x)u_{\lambda_k}^2d\mu
\geq \frac{p-1}{2(p+1)}\lambda_k a(x_0)\mu(x_0)u_{\lambda_k}(x_0)^2
$$
Since $a(x_0)=a_0>0$, $\mu(x_0)\geq\mu_{\min}>0$, $u_{\lambda_k}(x_0)\rightarrow u_0(x_0)=u_0\neq 0$ and $\lambda_k\rightarrow \infty$, we get
$$
\lim_{k\rightarrow \infty}J_{\lambda_k}(u_{\lambda_k})=\infty
$$
which is a contradiction to the fact that $m_{\lambda_k}< m_{\Omega}$.

Since for any $2\leq q< \infty$,
$$
u_{\lambda_k}\rightarrow u_0 \ \text{in}\ L^q(V)
$$
and
$$
u_{\lambda_k}\rightharpoonup u_0 \ \text{in}\ W^{1,2}(V),
$$
we get
\begin{eqnarray*}
\int_{\Omega\cup \partial\Omega}(|\nabla u_0|^2+u_0^2) d\mu&\leq&\int_V(|\nabla u_0|^2+u_0^2) d\mu\\
&\leq& \liminf_{k\rightarrow \infty}\int_V(|\nabla u_{\lambda_k}|^2+u_{\lambda_k}^2) d\mu\\
&\leq& \liminf_{k\rightarrow \infty}\int_V(|\nabla u_{\lambda_k}|^2+(\lambda_k a+1)u_{\lambda_k}^2) d\mu\\
&=& \liminf_{k\rightarrow \infty}\int_V |u|_{\lambda_k}^{p+1} d\mu\\
&=&\int_V |u_0|^{p+1} d\mu
\end{eqnarray*}
Noticing that $u_0|_{\Omega^c}=0$, we get
$$
\int_{\Omega\cup \partial\Omega}(|\nabla u_0|^2+u_0^2) d\mu\leq\int_\Omega |u_0|^{p+1} d\mu.
$$
Then there exists $\alpha\in (0,1]$ such that $\alpha u_0\in \mathcal {N}_\Omega$, i.e.
$$
\int_{\Omega\cup \partial\Omega} |\alpha\nabla u_0|^2+|\alpha u_0|^2 d\mu=\int_\Omega |\alpha u_0|^{p+1}d\mu.
$$
This implies that
\begin{eqnarray*}
J_\Omega(\alpha u_0)&=&\frac{p-1}{2(p+1)}\int_{\Omega\cup \partial\Omega} (|\alpha \nabla u_0|^2+|\alpha u_0|^2)d\mu\\
&\leq&\frac{p-1}{2(p+1)}\int_V (|\alpha \nabla u_0|^2+|\alpha u_0|^2)d\mu\\
&\leq&\frac{p-1}{2(p+1)}\int_V (|\nabla u_0|^2+|u_0|^2)d\mu\\
&\leq&\liminf_{k\rightarrow \infty}\frac{p-1}{2(p+1)}\int_V (|\nabla u_{\lambda_k}|^2+(\lambda_k a(x)+1)|u_{\lambda_k}|^2)d\mu\\
&=&M.
\end{eqnarray*}
Consequently, $M\geq m_\Omega$. Then we get that
$$\lim_{\lambda\rightarrow \infty}m_{\lambda}=m_\Omega.$$
\end{proof}

Next we prove Theorem \ref{convergence}.

\begin{proof}
We need to prove that for any sequence $\lambda_k\rightarrow \infty$, the corresponding $u_{\lambda_k}\in \mathcal{N}_{\lambda_k}$ satisfying $J_{\lambda_k}(u_{\lambda_k})=m_{\lambda_k}$ converges in $W^{1,2}(V)$ to a ground state solution $u_\Omega$ of (\ref{dirichlet}) along a subsequence.

Lemma \ref{bound1} gives that $u_{\lambda_k}$ is bounded in $E_{\lambda_k}$ and the upper-bound is independent of $\lambda_k$. Consequently, we have that $\{u_{\lambda_k}\}$ is also bounded in $W^{1,2}(V)$. Therefore, we can assume that for any $2\leq q< \infty$,
$$
u_{\lambda_k}\rightarrow u_0 \ \text{in}\ L^q(V)
$$
and
$$
u_{\lambda_k}\rightharpoonup u_0 \ \text{in}\ W^{1,2}(V).
$$
Moreover, we get from Lemma \ref{bound} that $u_0\not\equiv 0$. As what we have done in Lemma \ref{limit}, we can prove that $u_0|_{\Omega^c}=0$. Then it is sufficient to show that as $k\rightarrow \infty$, we have
$$
\lambda_k\int_V a(x)u_{\lambda_k}^2d\mu\rightarrow 0
$$
and
$$
\int_V |\nabla u_{\lambda_k}|^2d\mu\rightarrow \int_V |\nabla u_0|^2d\mu.
$$
If
$$
\lim_{k\rightarrow\infty}\lambda_k\int_V a(x)u_{\lambda_k}^2d\mu=\delta>0
$$
we have
\begin{eqnarray*}
\int_{\Omega\cup\partial\Omega} |\nabla u_0|^2+u_0^2d\mu&<&\int_V |\nabla u_0|^2+u_0^2d\mu+\delta\\
&\leq&\liminf_{k\rightarrow \infty}\int_V |\nabla u_{\lambda_k}|^2+(\lambda_k a(x)+1)u_{\lambda_k}^2d\mu\\
&=&\liminf_{k\rightarrow \infty}\int_V |u_{\lambda_k}|^{p+1}d\mu\\
&=&\int_\Omega |u_0|^{p+1}d\mu
\end{eqnarray*}
Then there exists $\alpha\in (0,1)$ such that $\alpha u_0\in \mathcal{N}_\Omega$.
Similarly, if
$$
\liminf_{k\rightarrow \infty}\int_V |\nabla u_{\lambda_k}|^2d\mu>\int_V |\nabla u_0|^2d\mu,
$$
we also have $\int_{\Omega\cup\partial \Omega} |\nabla u_0|^2+u_0^2d\mu<\int_\Omega |u_0|^{p+1}d\mu$. Then in both cases, we can find $\alpha\in (0,1)$ such that $\alpha u_0\in \mathcal{N}_\Omega$. Consequently, we have
\begin{eqnarray*}
J_\Omega(\alpha u_0)&=&\frac{p-1}{2(p+1)}\int_\Omega (|\alpha \nabla u_0|^2+|\alpha u_0|^2)d\mu\\
&=&\frac{p-1}{2(p+1)}\alpha^2\int_\Omega (|\nabla u_0|^2+|u_0|^2)d\mu\\
&<&\frac{p-1}{2(p+1)}\int_V (|\nabla u_0|^2+| u_0|^2)d\mu\\
&\leq&\lim_{k\rightarrow \infty}\frac{p-1}{2(p+1)}\int_V (|\nabla u_{\lambda_k}|^2+(\lambda_k a(x)+1) u_{\lambda_k}^2)d\mu\\
&=&\lim_{k\rightarrow \infty}J_{\lambda_k}(u_{\lambda_k})\\
&=&m_\Omega
\end{eqnarray*}
which is a contradiction. Then we have that $u_0$ is a solution of (\ref{dirichlet}) and Lemma \ref{limit} gives that in fact $u_0$ is a ground state solution of (\ref{dirichlet}).
\end{proof}

\section{Numerical experiments}

To illustrate our results, we consider a finite connected graph $G_9=(V,E)$ with $9$ vertexes which is shown in Figure \ref{graph}.
\begin{figure}[h]
\begin{centering}
\includegraphics[width=180pt]{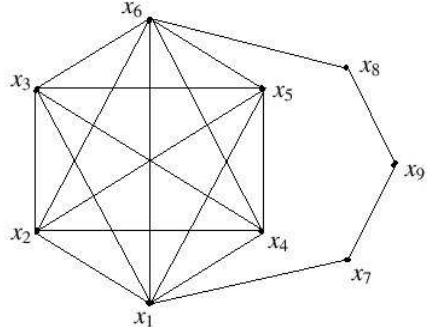}\\
\end{centering}
\caption[dataset]{The graph $G_9$}\label{graph}
\end{figure}
For $G_9$, the vertex set $V$ is $\{x_1,x_2,\cdots, x_9\}$. If two vertexes $x_i$ and $x_j$ are connected by an edge, we denote the edge by $x_{ij}$. The edge set $E$ of $G_9$ is composed of $x_{12},x_{13}, x_{14}, x_{15},x_{16},x_{17},x_{23},x_{24},x_{25},x_{26},x_{34},
x_{35},x_{36},x_{45},x_{46},x_{56},x_{68},x_{79},
x_{89}$. For simplicity, we set the measure $\mu(x_i)=1$ for $i=1,2,\cdots, 9$ and and set $\omega_{ij}=\omega_{ji}=1$ if $x_{ij}\in E$. Obviously, $G_9$ is a finite and connected graph with symmetric weight and uniformly positive measure and is suitable for us to do numerical experiments on it. Now we consider the equation (\ref{equation}) on $G_9$ with $p=2$ and let the potential to be
\begin{align*}
a(x_i)=\begin{cases}
0\ &\text{for}\  1\leq i\leq 6;\\
1\ &\text{for}\  i=7,8,9.
\end{cases}
\end{align*}
Under these assumptions, we have that the potential well $\Omega=\{x_1,x_2,\cdots, x_6\}$ and its boundary $\partial \Omega=\{x_7,x_8\}$.

With the help of the random search method, we start iterations from the initial value in Table \ref{initial value} to get a numerical solution of (\ref{equation}). In this table, we use $u_i$ to denote $u(x_i)$ for $i=1,2,\cdots,9$.
\begin{table}[h]
\caption{Initial value}\label{initial value}
\begin{tabular}{|c|c|c|c|c|c|c|c|c|}
  \hline
  $u_1$ & $u_2$ & $u_3$ & $u_4$ & $u_5$ & $u_6$ & $u_7$ & $u_8$ & $u_9$ \\
  \hline
  8.1472 & 9.0579 & 1.2699 & 9.1338 & 6.3236 & 0.9754 & 2.7850 & 5.4688 & 9.5751 \\
  \hline
\end{tabular}
\end{table}
After computations by MATLAB, we get the corresponding numerical solution $u_{\lambda}$ and find that the values of the solution $u_{\lambda}(x)$ at $x_7$, $x_8$ and $x_9$ decrease as $\lambda$ increasing and almost equal to $0$ after we take $\lambda$ bigger than $10^8$. These coincide with the conclusions in Theorem \ref{convergence} and the details are shown in Figure \ref{numerical solution}.
\begin{figure}[ht]
\centering
{
\subfigure[Trend of $u_7$]{\includegraphics[width=6.5cm,clip]{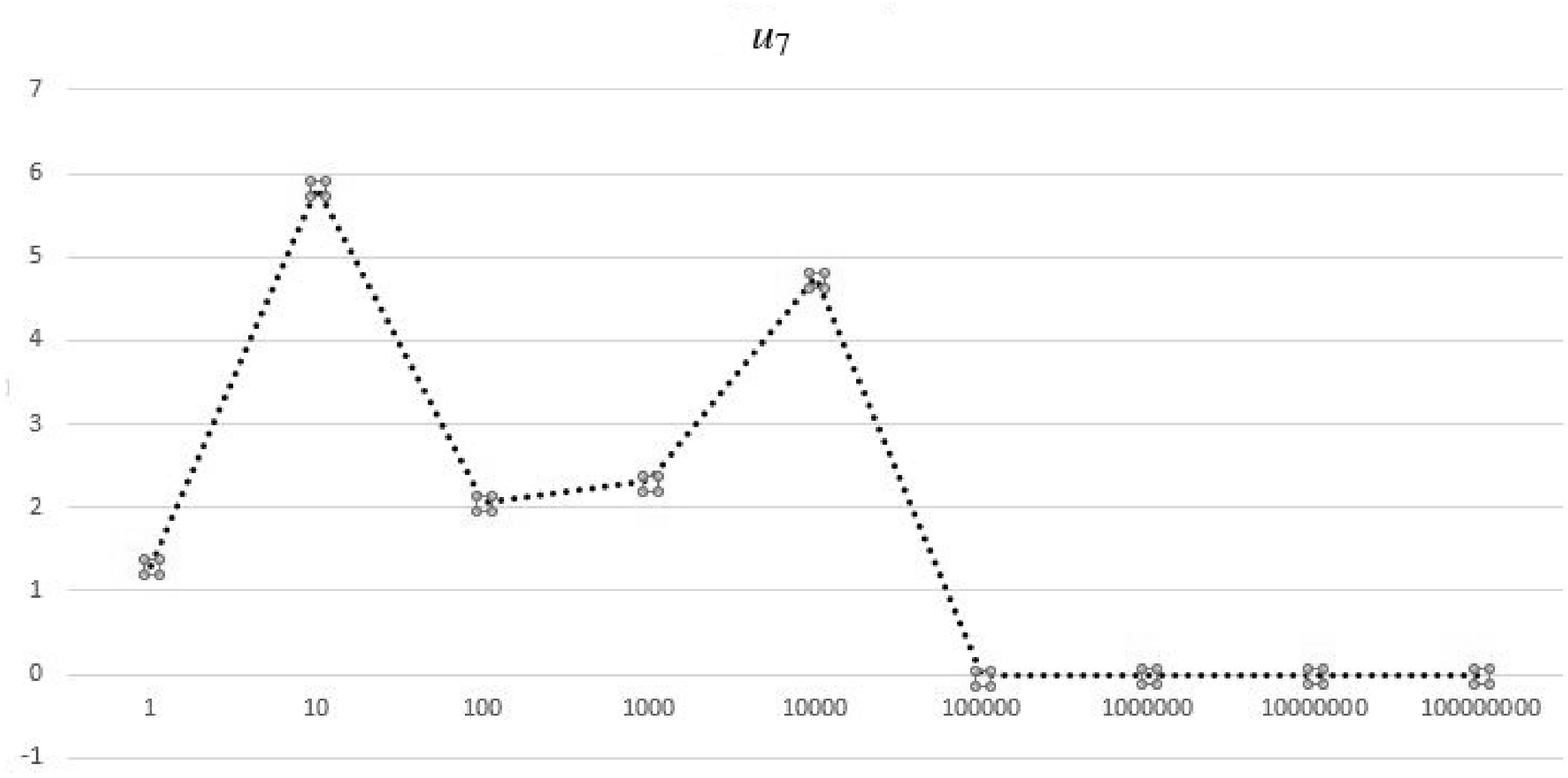}}
\subfigure[Trend of $u_8$]{\includegraphics[width=6.5cm,clip]{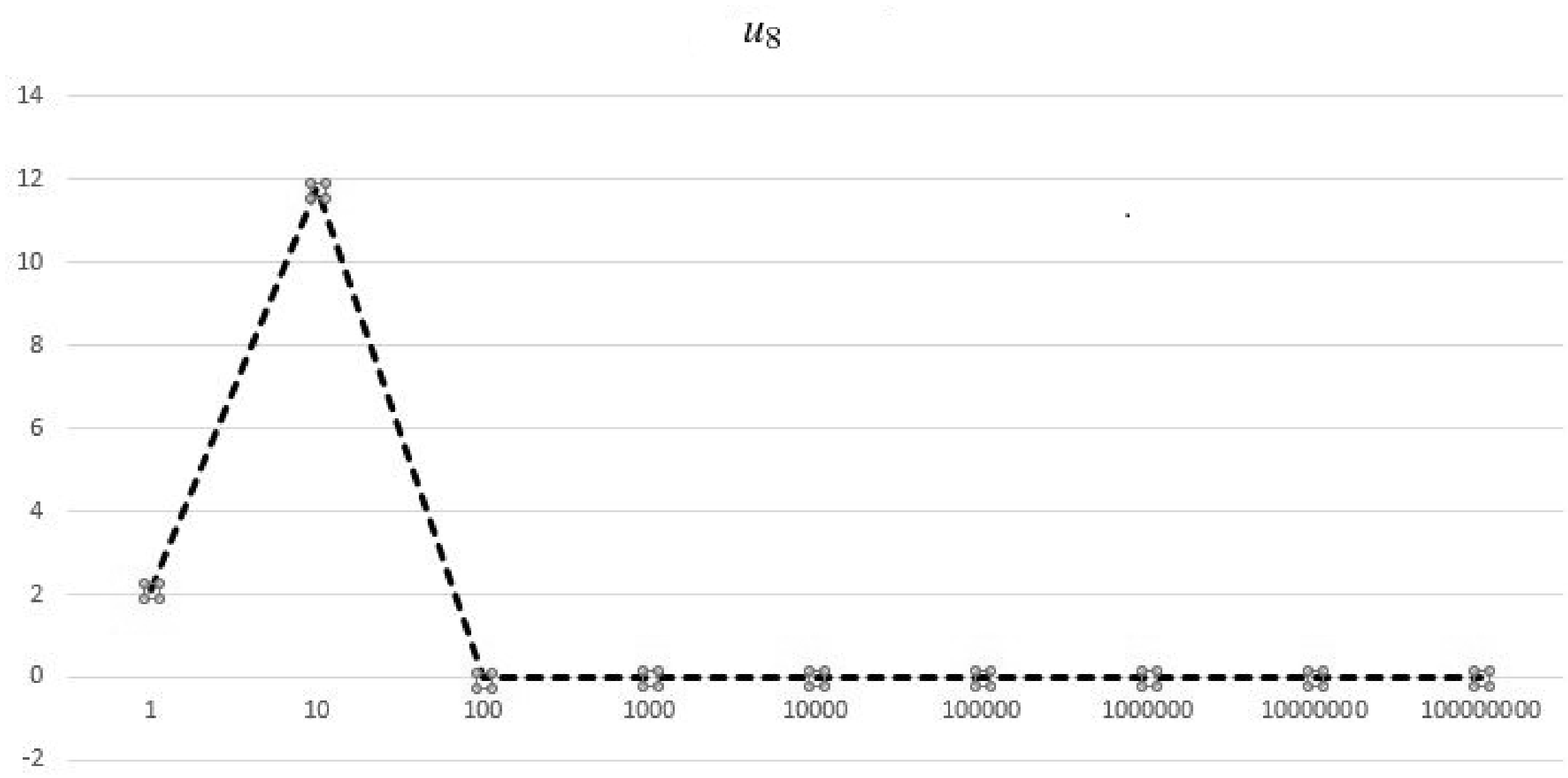}}
\subfigure[Trend of $u_9$]{\includegraphics[width=6.5cm,clip]{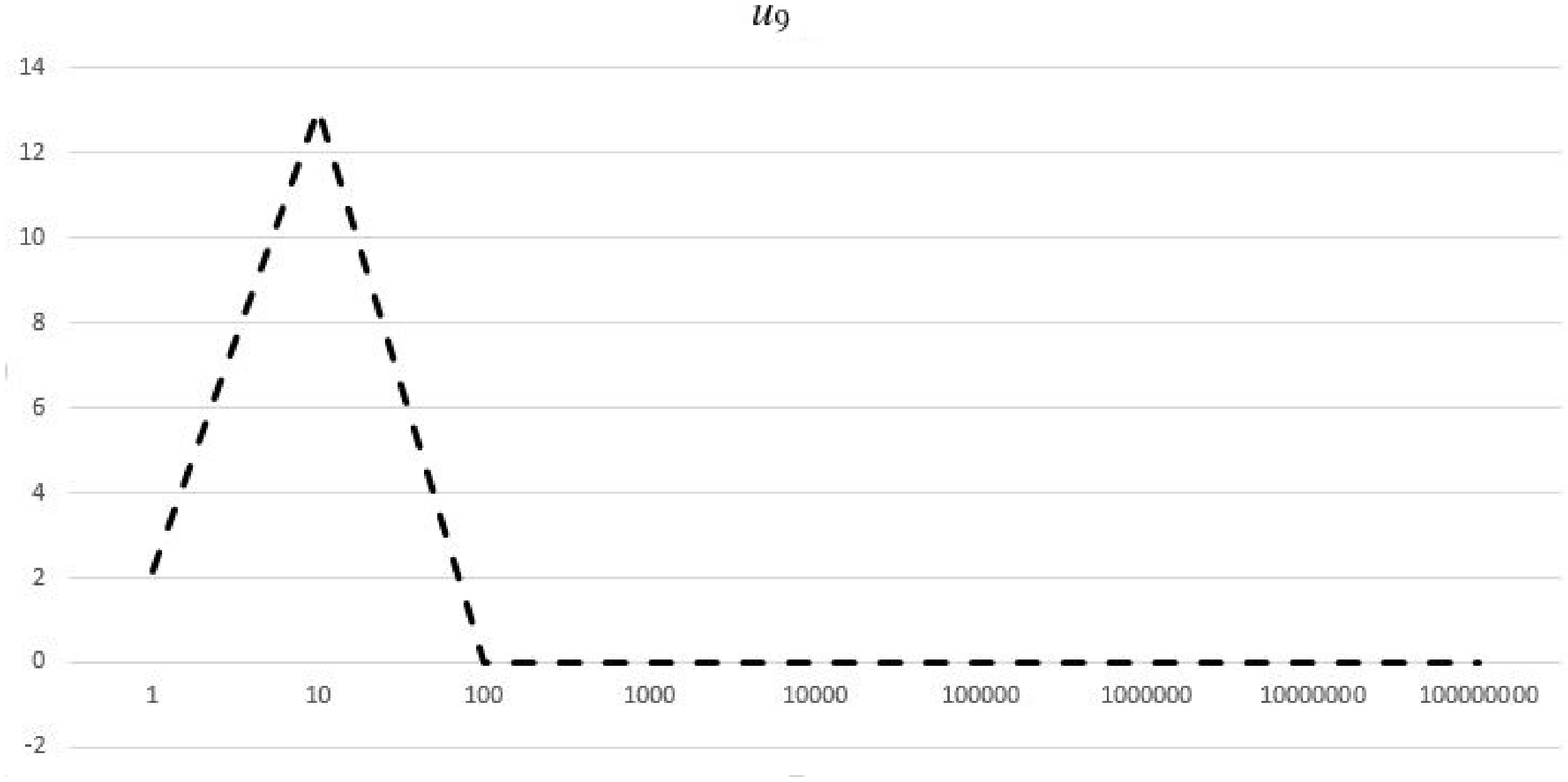}}
\subfigure[Trend of $u(x)$]{\includegraphics[width=6.5cm,clip]{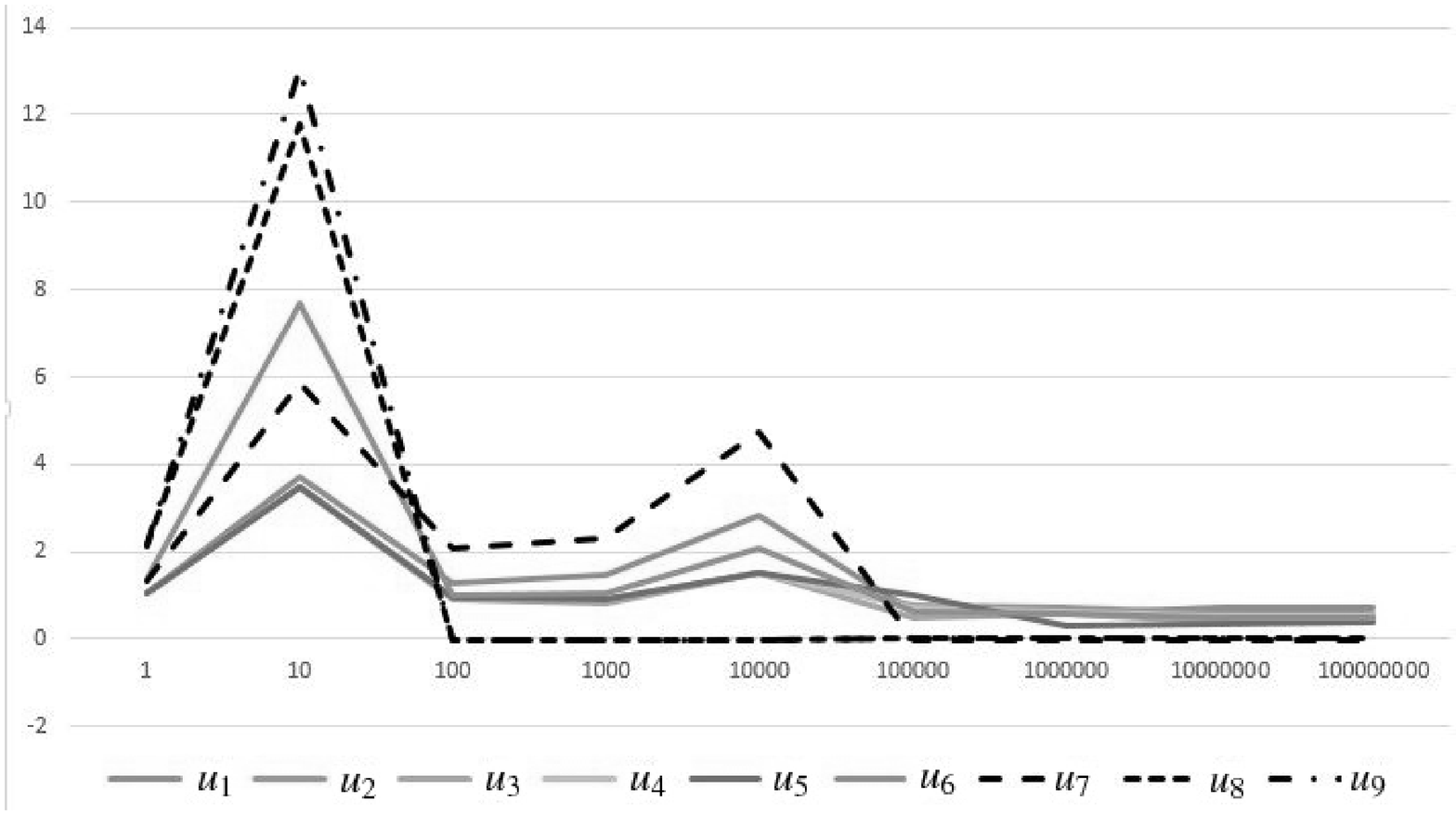}}
}
\caption{Trend of the numerical solutions}\label{numerical solution}
\end{figure}

{\bf Acknowledgements.} The first author is supported by the Funding of Beijing Philosophy and Social Science No. 15JGC153 and the MOE Project of Humanities and Social Sciences No. 16YJCZH148. The second author is supported the Fundamental Research Funds for the Central Universities. Both of the authors are supported by the MOE Project of Key Research Institute of Humanities and Social Sciences at Universities No.16JJD790060 and they would like to thank members of Data Lighthouse for their helpful conversations and valuable suggestions.

\end{document}